\documentclass{amsart}

\usepackage{amsthm}
\usepackage[leqno]{amsmath}
\usepackage[all]{xy} \SelectTips{eu}{}
\usepackage{latexsym,amsfonts,amssymb}
\usepackage{hyperref}

\newcommand{\numberseries}{\bfseries}   

\newlength{\thmtopspace}                
\newlength{\thmbotspace}                
\newlength{\thmheadspace}               
\newlength{\thmindent}                  

\newtheoremstyle{fixed bf head,slanted body}
                {\thmtopspace}{\thmbotspace}{\slshape}
                {\thmindent}{\bfseries}{.}{\thmheadspace}
                {{\numberseries \thmnumber{#2\;}}\thmname{#1}\thmnote{ (#3)}}

\newtheoremstyle{fixed bf head,upright body}
                {\thmtopspace}{\thmbotspace}{\upshape}
                {\thmindent}{\bfseries}{.}{\thmheadspace}
                {{\numberseries \thmnumber{#2\;}}\thmname{#1}\thmnote{ (#3)}}

\newtheoremstyle{numbered paragraph}
                {\thmtopspace}{\thmbotspace}{\upshape}
                {\thmindent}{\upshape}{}{\thmheadspace}
                {{\numberseries \thmnumber{#2.}}}

\theoremstyle{fixed bf head,slanted body}
\newtheorem{res}{}[section]
\newtheorem{thm}[res]{Theorem}          \newtheorem*{thm*}{Theorem}
\newtheorem{prp}[res]{Proposition}      \newtheorem*{prp*}{Proposition}
\newtheorem{cor}[res]{Corollary}        \newtheorem*{cor*}{Corollary}
\newtheorem{lem}[res]{Lemma}            \newtheorem*{lem*}{Lemma}

\theoremstyle{fixed bf head,upright body}
\newtheorem{dfn}[res]{Definition}       \newtheorem*{dfn*}{Definition}
\newtheorem{con}[res]{Construction}     \newtheorem*{con*}{Construction}
      \newtheorem*{obs*}{Observation}
\newtheorem{rmk}[res]{Remark}           \newtheorem*{rmk*}{Remark}
\newtheorem{exa}[res]{Example}          \newtheorem*{exa*}{Example}

\theoremstyle{numbered paragraph}
\newtheorem{ipg}[res]{}

\newlength{\thmlistleft}        
\newlength{\thmlistright}       
\newlength{\thmlistpartopsep}   
\newlength{\thmlisttopsep}      
\newlength{\thmlistparsep}      
\newlength{\thmlistitemsep}     

\newcounter{eqc} 
\newenvironment{eqc}{\begin{list}{\upshape (\textit{\roman{eqc}})}%
    {\usecounter{eqc}%
      \setlength{\leftmargin}{\thmlistleft}%
      \setlength{\labelwidth}{\thmlistleft}%
      \setlength{\rightmargin}{\thmlistright}%
      \setlength{\partopsep}{\thmlistpartopsep}%
      \setlength{\topsep}{\thmlisttopsep}%
      \setlength{\parsep}{\thmlistparsep}%
      \setlength{\itemsep}{\thmlistitemsep}}}%
  {\end{list}}%

\newcommand{\eqclbl}[1]{{\upshape(\textit{#1})}}

\newcounter{prt}
\newenvironment{prt}{\begin{list}{\upshape (\alph{prt})}%
    {\usecounter{prt}%
      \setlength{\leftmargin}{\thmlistleft}%
      \setlength{\labelwidth}{\thmlistleft}%
      \setlength{\rightmargin}{\thmlistright}%
      \setlength{\partopsep}{\thmlistpartopsep}%
      \setlength{\topsep}{\thmlisttopsep}%
      \setlength{\parsep}{\thmlistparsep}%
      \setlength{\itemsep}{\thmlistitemsep}}}%
  {\end{list}}%

\newcommand{\prtlbl}[1]{{\upshape(#1)}}

  \newcommand{\proofofimp}[3][:]{\mbox{\eqclbl{#2}$\!\implies\!$\eqclbl{#3}#1}}
  \newcommand{\proofoftag}[2][:]{(#2)#1}

\newcommand{\pgref}[1]{\ref{#1}}
\newcommand{\thmref}[2][Theorem~]{#1\pgref{thm:#2}}
\newcommand{\corref}[2][Corollary~]{#1\pgref{cor:#2}}
\newcommand{\prpref}[2][Proposition~]{#1\pgref{prp:#2}}
\newcommand{\lemref}[2][Lemma~]{#1\pgref{lem:#2}}

\newcommand{\conref}[2][Construction~]{#1\pgref{con:#2}}

\newcommand{\exaref}[2][Example~]{#1\pgref{exa:#2}}
\newcommand{\rmkref}[2][Remark~]{#1\pgref{rmk:#2}}
\newcommand{\secref}[2][Section~]{#1\ref{sec:#2}}
\renewcommand{\eqref}[1]{(\pgref{eq:#1})}
\newcommand{\thmcite}[2][?]{\cite[thm.~#1]{#2}}
\newcommand{\corcite}[2][?]{\cite[cor.~#1]{#2}}
\newcommand{\prpcite}[2][?]{\cite[prop.~#1]{#2}}
\newcommand{\lemcite}[2][?]{\cite[lem.~#1]{#2}}
\newcommand{\seccite}[2][?]{\cite[sec.~#1]{#2}}
\newcommand{\exacite}[2][?]{\cite[ex.~#1]{#2}}

\newcommand{\xycomma}[1][,]{\rlap{\;#1}}
\newcommand{\fg}{finitely generated}
\newcommand{\fp}{finitely presented}

\newcommand{\submod}[2][R]{#1\langle#2\rangle}
\newcommand{\Id}[1]{1^{#1}}

\newcommand{\dxra}[2][]{\xrightarrow[#1]{\mspace{2mu}#2\mspace{2mu}}}

\newcommand{\e}{\varepsilon}
\renewcommand{\k}{\kappa}

\renewcommand{\l}{\lambda}

\newcommand{\Ct}[2][u]{\operatorname{D}^{#1}(#2)}
\def\urltilda{\kern -.15em\lower .7ex\hbox{\~{}}\kern .04em} 

\newcommand{\setof}[3][\mspace{1mu}]{\{#1#2 \mid #3#1\}}

\newcommand{\NN}{\mathbb{N}}
\newcommand{\ZZ}{\mathbb{Z}}
\newcommand{\dinZ}{{\d\in\ZZ}}
\newcommand{\qtext}[1]{\quad\text{#1}\quad}
\newcommand{\qqtext}[1]{\qquad\text{#1}\qquad}
\newcommand{\qand}{\qtext{and}}
\newcommand{\qqand}{\qqtext{and}}
\newcommand{\deq}{\:=\:}
\newcommand{\dis}{\:\is\:}
\newcommand{\finsum}{\bigoplus}

\DeclareMathOperator*{\dcoprod}{\textstyle\coprod}
 \renewcommand{\a}{\alpha}
 \renewcommand{\b}{\beta}
 \renewcommand{\d}{v}            
 \newcommand{\f}{\varphi}
 \newcommand{\g}{\gamma}
 \renewcommand{\l}{\lambda}
\newcommand{\s}{\sigma}
 
 \newcommand{\Fu}[1][F]{\operatorname{#1}}
 \newcommand{\blank}{{\scriptstyle\stackrel{-}{}}}
 \newcommand{\is}{\cong}
 \newcommand{\qis}{\simeq}
 \renewcommand{\le}{\leqslant}
 \renewcommand{\ge}{\geqslant}
 \newcommand{\onto}{\twoheadrightarrow}
 \newcommand{\into}{\hookrightarrow}
\newcommand{\lra}{\longrightarrow}
\newcommand{\xra}[2][]{\xrightarrow[#1]{\;#2\;}}
\newcommand{\ira}{\xra{\;\is\;}}
\newcommand{\qra}{\xra{\;\qis\;}}
\newcommand{\Rop}{R^\circ}
\newcommand{\mapdef}[4][\rightarrow]{\nobreak{#2\colon #3 #1 #4}}
\newcommand{\qisdef}[4][\xra{\qis}]{\nobreak{#2\colon #3 #1 #4}}
\newcommand{\dmapdef}[4][\lra]{\nobreak{#2\colon #3\:#1\:#4}}
\renewcommand{\Im}[1]{\nobreak{\operatorname{Im}#1}}
\newcommand{\Ker}[1]{\nobreak{\operatorname{Ker}#1}}
\newcommand{\Coker}[1]{\nobreak{\operatorname{Coker}#1}}
\newcommand{\Cone}[1]{\nobreak{\operatorname{Cone}#1}}
\newcommand{\tev}[1]{\omega^{#1}}
\newcommand{\hev}[1]{\theta^{#1}}
\newcommand{\bid}[2]{\delta^{#1}_{#2}}
\newcommand{\sign}[1]{(-1)^{#1}}
\newcommand{\dgr}[1]{|#1|}
\newcommand{\dif}[2][]{{\partial}^{#2}_{#1}}
\newcommand{\Bo}[2][]{\operatorname{B}_{#1}(#2)}
\newcommand{\Cy}[2][]{\operatorname{Z}_{#1}(#2)}
\newcommand{\Co}[2][]{\operatorname{C}_{#1}(#2)}
\renewcommand{\H}[2][]{\operatorname{H}_{#1}(#2)}
\newcommand{\Shift}[2][]{\mathsf{\Sigma}^{#1}{#2}}
\newcommand{\Hom}[3][R]{\operatorname{Hom}_{#1}(#2,#3)}
\newcommand{\tp}[3][R]{\nobreak{#2\otimes_{#1}#3}}
\DeclareMathOperator*{\colim}{colim}
\newcommand{\Catfont}[1]{\mathcal{#1}}
\newcommand{\Cat}[2]{{\mathcal{#2}}(#1)}
\newcommand{\C}[1][R]{\Cat{#1}{C}}
\newcommand{\GrMod}[1][R]{\Cat{#1}{GM}}

\makeatletter
\def\@nobreak@#1{\mathchoice%
  {\nobreakdef@\displaystyle\f@size{#1}}%
  {\nobreakdef@\nobreakstyle\tf@size{\firstchoice@false #1}}%
  {\nobreakdef@\nobreakstyle\sf@size{\firstchoice@false #1}}%
  {\nobreakdef@\nobreakstyle\ssf@size{\firstchoice@false #1}}%
  \check@mathfonts}%
\def\nobreakdef@#1#2#3{\hbox{{%
                    \everymath{#1}%
                    \let\f@size#2\selectfont%
                    #3}}}%
\makeatother

\begin{document}

\title{The direct limit closure of perfect complexes}

\author[L.\,W. Christensen]{Lars Winther Christensen}

\address{Texas Tech University, Lubbock, TX 79409, U.S.A.}

\email{lars.w.christensen@ttu.edu}

\urladdr{http://www.math.ttu.edu/\urltilda lchriste}

\author[H. Holm]{Henrik Holm}

\address{University of Copenhagen, 2100 Copenhagen {\O}, Denmark}
 
\email{holm@math.ku.dk}

\urladdr{http://www.math.ku.dk/\urltilda holm/}

\thanks{This work was partly supported by NSA grant H98230-11-0214
  (L.W.C.).}

\date{11 June 2013}

\dedicatory{To Hans-Bj\o rn Foxby---our teacher, colleague, and
  friend}




\begin{abstract}
  Every projective module is flat. Conversely, every flat module is a
  direct limit of finitely generated free modules; this was proved
  independently by Govorov and Lazard in the 1960s. In this paper we
  prove an analogous result for complexes of modules, and as
  applications we reprove some results due to Enochs and Garc\'ia
  Rozas and to Neeman.
\end{abstract}

\maketitle

\section{Introduction}

\noindent
Let $R$ be a ring. In contrast to the projective objects in the
category of $R$-modules, i.e.\ the projective $R$-modules, the
projective objects in the category of $R$-complexes are not of much
utility; indeed, they are nothing but contractible (split) complexes
of projective $R$-modules.  In the category of complexes, the relevant
alternative to projectivity---from the homological point of view, at
least---is semi-projectivity. A complex $P$ is called
\emph{semi-projective} (or \emph{DG-projective}) if the the total Hom
functor $\Hom[]{P}{\blank}$ preserves surjective quasi-isomorphisms,
i.e.~surjective morphisms that induce isomorphisms in homology. The
semi-projective complexes are exactly the cofibrant objects in the
standard model structure on the category of complexes; see Hovey
\cite[\S2.3]{modcat}.  Alternatively, a complex is semi-projective if
and only if it consists of projective modules and it is K-projective
in the sense of Spaltenstein~\cite{NSp88}. The notion of
semi-projectivity in the category of complexes extends the notion of
projectivity in the category of modules in a natural and useful way: A
module is projective if and only if it is semi-projective when viewed
as a complex.

Similarly, a complex $F$ is \emph{semi-flat} if the total tensor
product functor $\tp[]{\blank}{F}$ preserves injective
quasi-isomorphisms; equivalently, $F$ is a complex of flat modules and
K-flat in the sense of \cite{NSp88}.  A module is flat if and only if
it is semi-flat when viewed as a complex.  Every semi-projective
complex is semi-flat, and simple examples of semi-projective complexes
are bounded complexes of \fg\ projective modules, also known as
\emph{perfect} complexes. The class of semi-flat complexes is closed
under direct limits, and our main result, \thmref{GL} below, shows
that every semi-flat complex is a direct limit of perfect
complexes. For modules, the theorem specializes to a classic result,
proved independently by Govorov~\cite{VEG65} and Lazard~\cite{DLz69}:
Every flat module is a direct limit of \fg\ free modules.

\pagebreak\begin{thm}
  \label{thm:GL}
  For an $R$-complex $F$ the following conditions are equivalent.
  \begin{eqc}
  \item $F$ is semi-flat.
  \item Every morphism of $R$-complexes $\mapdef{\f}{N}{F}$ with $N$
    bounded and degreewise finitely presented admits a factorization,
    \begin{equation*}
      \xymatrix@!=0.7pc{
        N \ar[dr]_-{\kappa} \ar[rr]^-{\f} & & F \\
        & L\xycomma[,] \ar[ur]_-{\lambda} & 
      }
    \end{equation*}
    where $L$ is a bounded complex of finitely generated free
    $R$-modules.
  \item There exists a set $\{L^u\}_{u \in U}$ of bounded complexes of
    finitely generated free $R$-modules and a pure epimorphism
    $\coprod_{u\in U}L^u \to F$.
  \item $F$ is isomorphic to a filtered colimit of bounded complexes
    of finitely generated free $R$-modules.
  \item $F$ is isomorphic to a direct limit of bounded complexes of
    finitely generated free $R$-modules.
  \end{eqc}
\end{thm}

\noindent
The theorem is proved in \secref{main}. The terminology used in the
statement is clarified in the sections leading up to the proof. In
\secref{fp} we show that the finitely presented objects in the
category of complexes are exactly the bounded complexes of finitely
presented modules. Results of Breitsprecher~\cite{SBr70} and
Crawley-Boevey~\cite{WCB94} show that the category of complexes is
locally finitely presented,
see~\rmkref{C-is-locally-finitely-presented}. Therefore, the
equivalence of \eqclbl{ii}, \eqclbl{iii}, and \eqclbl{iv} follows from
\cite[(4.1)]{WCB94}. Furthermore, a result by Ad{\'a}mek and
Rosick{\'y} \cite[thm.~1.5]{AdamekRosicky} shows that \eqclbl{iv} and
\eqclbl{v} are equivalent for quite general reasons; thus our task is
to prove that the equivalent conditions \eqclbl{ii}--\eqclbl{v} are
also equivalent to \eqclbl{i}.

The characterization of semi-flat complexes in \thmref{GL} opens to a
study of the interplay between semi-flatness and purity in the
category of complexes; this is the topic of \secref{purity}. We show,
for example, that a complex $F$ is semi-flat if and only if every
surjective quasi-isomorphism $M \to F$ is a pure epimorphism. This
compares to Lazard's \corcite[1.3]{DLz69} which states that a module
$F$ is flat if and only if every surjective homomorphism $M \to F$ is
a pure epimorphism.

In the final \secref{8}, we use \thmref{GL} to reprove a few results
due to Enochs and Garc\'ia Rozas \cite{EEEJGR98} and to
Neeman~\cite{ANm08}; our proofs are substantially different from the
originals. In \thmref{contractible} we show that an acyclic semi-flat
complex is a direct limit of contractible perfect complexes. Combined
with a result of Benson and Goodearl~\cite{DJBKRG00} this enables us
to show in \thmref{sfp} that a semi-flat complex of projective modules
is semi-projective.

\section{Complexes}
\label{sec:complexes}

\noindent
In this paper $R$ is a ring, and the default action on modules is on
the left. Thus, $R$-modules are left $R$-modules, while right
$R$-modules are considered to be (left) modules over the opposite ring
$\Rop$. The definitions and results listed in this section are
standard and more details can be found in textbooks, such as Weibel's
\cite{Wei}, and in the paper \cite{LLAHBF91} by Avramov and Foxby.

An $R$-complex $M$ is a graded $R$-module $M = \coprod_{\dinZ}M_\d$
equipped with a differential, that is, an $R$-linear map
$\mapdef{\dif{M}}{M}{M}$ that satisfies $\dif{M}\dif{M}=0$ and
$\dif{M}(M_\d) \subseteq M_{\d-1}$ for every $\dinZ$. The homomorphism
$M_\d \to M_{\d-1}$ induced by $\dif{M}$ is denoted $\dif[\d]{M}$.
Thus, an $R$-complex $M$ can be visualized as follows,
\begin{equation*}
  M \deq 
  \cdots \lra M_{\d+1} \dxra{\dif[\d+1]{M}} M_\d  \dxra{\dif[\d]{M}} M_{\d-1}
  \lra \cdots\;.
\end{equation*}

The category of $R$-complexes is denoted $\C$. We identify the
category of graded $R$-modules with the full subcategory of $\C$ whose
objects are $R$-complexes with zero differential.

For an $R$-complex $M$ with differential $\dif{M}$, set
$\Cy{M}=\Ker{\dif{M}}$, $\Bo{M}=\Im{\dif{M}}$,
$\Co{M}=\Coker{\dif{M}}$, and $\H{M}=\Cy{M}/\Bo{M}$; they are sub-,
quotient, and subquotient complexes of $M$. Furthermore,
$\Cy{\blank}$, $\Bo{\blank}$, $\H{\blank}$ and $\Co{\blank}$ are
additive endofunctors on $\C$.

A complex $M$ with $\H{M}=0$ is called \emph{acyclic}. The
\emph{shift} of $M$ is the complex $\Shift{M}$ with $(\Shift{M})_\d =
M_{\d-1}$ and $\dif[\d]{\Shift{M}} = -\dif[\d-1]{M}$. A morphism
$\mapdef{\a}{M}{N}$ of complexes is called a \emph{quasi-isomorphism}
if $\mapdef{\H{\a}}{\H{M}}{\H{N}}$ is an isomorphism.

\begin{ipg}
  \label{cone}
  To a morphism $\mapdef{\a}{M}{N}$ of $R$-complexes one associates a
  complex $\Cone{\a}$, called the \emph{mapping cone} of $\a$; it fits
  into a degreewise split exact sequence,
  \begin{equation*}
    0 \lra N \lra \Cone{\a} \lra \Shift{M} \lra 0\;.
  \end{equation*}
  The morphism $\a$ is a quasi-isomorphism if and only if $\Cone{\a}$
  is acyclic.
\end{ipg}

\begin{ipg}
  \label{Hom}
  Let $M$ and $N$ be $R$-complexes. The total Hom complex, written
  $\Hom{M}{N}$, yields a functor
  \begin{equation*}
    \dmapdef{\Hom{\blank}{\blank}}{\C^\mathrm{op} \times \C}{\C[\ZZ]}\;.
  \end{equation*}
  The functor $\Hom{M}{\blank}$ commutes with mapping cones, that is,
  for every morphism $\a$ of $R$-complexes there is an isomorphism of
  $\ZZ$-complexes,
  \begin{equation*}
    \Cone{\Hom{M}{\a}} \dis \Hom{M}{\Cone{\a}}\;.
  \end{equation*}
  There is an equality of abelian groups,
  \begin{equation*}
    \Cy[0]{\Hom{M}{N}} \deq \C(M,N)\,,
  \end{equation*}
  where the right-hand side is the hom-set in the category $\C$.
\end{ipg}

\begin{ipg}
  \label{tp}
  Let $M$ be an $\Rop$-complex and let $N$ be an $R$-complex. The
  total tensor product complex, written $\tp{M}{N}$, yields a functor
  \begin{equation*}
    \dmapdef{\tp{\blank}{\blank}}{\C[\Rop] \times \C}{\C[\ZZ]}\;.
  \end{equation*}
  The functor $\tp{M}{\blank}$ commutes with mapping cones, that is,
  for every morphism $\a$ of $R$-complexes there is an isomorphism of
  $\ZZ$-complexes,
  \begin{equation*}
    \Cone{(\tp{M}{\a})} \dis \tp{M}{\Cone{\a}}\,.
  \end{equation*}
\end{ipg}

For a homogeneous element $m$ in a graded module (or a complex) $M$,
we write $\dgr{m}$ for its degree.

\begin{ipg}
  \label{bid}
  Let $M$ be an $R$-complex. The \emph{biduality} morphism
  \begin{gather*}
    \dmapdef{\bid{M}{}}{M}{\Hom[\Rop]{\Hom{M}{R}}{R}}\\
    \intertext{\text{is given by}} \bid{M}{}(m)(\psi) \deq
    \sign{\dgr{\psi}\dgr{m}}\psi(m)
  \end{gather*}
  for homogeneous elements $m \in M$ and $\psi \in \Hom{M}{R}$.

  The morphism $\bid{M}{}$ of $R$-complexes is an isomorphism if $M$
  is a complex of finitely generated projective $R$-modules.
\end{ipg}

\begin{ipg}
  \label{tev}
  Let $M$ and $N$ be $R$-complexes and let $X$ be a complex of
  $R$--$\Rop$-bimodules. The \emph{tensor evaluation} morphism
  \begin{gather*}
    \dmapdef{\tev{MXN}}{\tp{\Hom{M}{X}}{N}}{\Hom{M}{\tp{X}{N}}}\\
    \intertext{is given by} \tev{MXN}(\tp[]{\psi}{n})(m) \deq
    \sign{\dgr{m}\dgr{n}}\tp[]{\psi(m)}{n}
  \end{gather*}
  for homogeneous elements $\psi \in \Hom{M}{X}$, $n \in N$, and $m
  \in M$.

  The morphism $\tev{MXN}$ of $\ZZ$-complexes is an isomorphism if $M$
  is a bounded complex of finitely generated projective $R$-modules
  and $X=R$.
\end{ipg}

\begin{ipg}
  \label{hev}
  Let $M$ be an $R$-complex, let $N$ be an $\Rop$-complex, and let $X$
  be a complex of $R$--$\Rop$-bimodules. The \emph{homomorphism
    evaluation} morphism
  \begin{gather*}
    \dmapdef{\hev{XNM}}{\tp{\Hom[\Rop]{X}{N}}{M}}{\Hom[\Rop]{\Hom{M}{X}}{N}}\\
    \intertext{is given by} \hev{XNM}(\tp[]{\psi}{m})(\f) \deq
    \sign{\dgr{\f}\dgr{m}}\psi\f(m)
  \end{gather*}
  for homogeneous elements $\psi \in \Hom[\Rop]{X}{N}$, $m \in M$, and
  $\f \in \Hom{M}{X}$.

  The morphism $\hev{XNM}$ of $\ZZ$-complexes is an isomorphism if $M$
  is a bounded complex of finitely generated projective $R$-modules
  and $X=R$.
\end{ipg}

\section{Filtered colimits}
\label{sec:colim}

We refer to MacLane \seccite[IX.1]{Mac} for background on colimits.

\begin{dfn}
  Let $\Catfont{A}$ be a category.  By a \emph{filtered colimit} in
  $\Catfont{A}$ we mean the colimit of a functor
  $\mapdef{\Fu}{\Catfont{J}}{\Catfont{A}}$, which is denoted
  $\colim_{J \in \Catfont{J}}\Fu(J)$, where $\Catfont{J}$ is a
  skeletally small filtered category. We reserve the term \emph{direct
    limit} for the colimit of a direct system, i.e.\ of a functor
  $\Catfont{J} \to \Catfont{A}$ where $\Catfont{J}$ is the filtered
  category associated to a directed set, i.e.\ a filtered preordered
  set.
\end{dfn}

Notice that some authors, including Crawley-Boevey \cite{WCB94}, use
the term ``direct limit'' for any filtered colimit. For a direct
system $\{A^u \to A^v\}_{u \le v}$ it is customary to write
$\varinjlim A^u$ for its direct limit, i.e.~its colimit, however, we
shall stick to the notation $\colim A^u$.

Let $\Catfont{A}$ and $\Catfont{B}$ be categories that have (all)
filtered colimits. Recall that a functor
$\mapdef{\Fu[T]}{\Catfont{A}}{\Catfont{B}}$ is said to \emph{preserve
  (filtered) colimits} if the canonical morphism in $\Catfont{B}$,
\begin{equation*}
  \colim_{J \in \Catfont{J}}\Fu[T](\Fu(J)) \lra
  \Fu[T](\colim_{J \in \Catfont{J}}\Fu(J))\,,
\end{equation*}
is an isomorphism for every (filtered) colimit $\colim_{J \in
  \Catfont{J}}\Fu(J)$ in $\Catfont{A}$.

We need a couple of facts about filtered colimits of complexes; the
arguments are given in \cite[lem.~2.6.14 and thm.~2.6.15]{Wei}.

\begin{ipg}
  \label{colim}
  The following assertions hold.
  \begin{prt}
  \item Every homogeneous element in a filtered colimit, $\colim_{J
      \in \Catfont{J}}\Fu(J)$, in $\C$ is in the image of the
    canonical morphism $\Fu(J) \to \colim_{J \in \Catfont{J}}\Fu(J)$
    for some $J \in \Catfont{J}$.
  \item Filtered colimits in $\C$ are exact (colimits are always right
    exact).
  \end{prt}
\end{ipg}

\begin{lem}
  \label{lem:2-3-preserve-filtered-colimit}
  \hspace{-1pt}Let $\Catfont{A}$ be a category with filtered colimits
  and let $\mapdef{\Fu[T]',\Fu[T],\Fu[T]''}{\Catfont{A}}{\C}$ be
  functors. The following assertions hold.
  \begin{prt}
  \item If\, \mbox{$0 \to \Fu[T]' \to \Fu[T] \to \Fu[T]''$} is an
    exact sequence and if $\Fu[T]$ and $\Fu[T]''$ preserve filtered
    colimits, then $\Fu[T]'$ preserves filtered colimits.
  \item If\, \mbox{$\Fu[T]' \to \Fu[T] \to \Fu[T]'' \to 0$} is an
    exact sequence and if $\Fu[T]'$ and $\Fu[T]$ preserve filtered
    colimits, then $\Fu[T]''$ preserves filtered colimits.
  \item If \mbox{$0 \to \Fu[T]' \to \Fu[T] \to \Fu[T]'' \to 0$} is an
    exact sequence and if $\Fu[T]'$ and $\Fu[T]''$ preserve filtered
    colimits, then $\Fu[T]$ preserves filtered colimits.
  \end{prt}
\end{lem}

\begin{proof}
  Let $\Catfont{J}$ be a skeletally small filtered category and let
  $\mapdef{\Fu}{\Catfont{J}}{\Catfont{A}}$ be a functor.
  
  \proofoftag{a} Exactness of the sequence \mbox{$0 \to \Fu[T]' \to
    \Fu[T] \to \Fu[T]''$} and left exactness of filtered colimits in
  $\C$ yield the following commutative diagram,
  \begin{equation*}
    \xymatrix{
      0 \ar[r] & 
      \displaystyle\colim_{J \in \Catfont{J}}\Fu[T]'(\Fu(J)) \ar[r] 
      \ar[d]^-{\mu'} &
      \displaystyle\colim_{J \in \Catfont{J}}\Fu[T](\Fu(J)) \ar[r] 
      \ar[d]^-{\mu} &
      \displaystyle\colim_{J \in \Catfont{J}}\Fu[T]''(\Fu(J)) 
      \ar[d]^-{\mu''} \\
      0 \ar[r] & 
      \displaystyle\Fu[T]'(\colim_{J \in \Catfont{J}}\Fu(J)) \ar[r] &
      \displaystyle\Fu[T](\colim_{J \in \Catfont{J}}\Fu(J)) \ar[r] &
      \displaystyle\Fu[T]''(\colim_{J \in \Catfont{J}}\Fu(J))\xycomma 
    }
  \end{equation*}
  where $\mu'$, $\mu$, and $\mu''$ are the canonical morphisms. If
  $\mu$ and $\mu''$ are isomorphisms, then so is $\mu'$ by the Five
  Lemma.

  Parts \prtlbl{b} and \prtlbl{c} have similar proofs.
\end{proof}

\begin{prp}
  \label{prp:ZCBH}
  The functors $\mapdef{\Fu[Z],\Fu[C],\Fu[B],\Fu[H]}{\C}{\C}$ preserve
  filtered colimits.
\end{prp}

\begin{proof}
  A graded $R$-module is considered as an $R$-complex with zero
  differential; let $\GrMod$ be the full subcategory of $\C$ whose
  objects are all graded $R$-modules.  Since the inclusion functor
  $\mapdef{\Fu[i]}{\GrMod}{\C}$ preserves filtered colimits, it
  suffices to argue that $\Fu[C]$ preserves filtered colimits when
  viewed as a functor from $\C$ to $\GrMod$. However, this functor
  $\Fu[C]$ has a right adjoint, namely the inclusion functor $\Fu[i]$,
  so it follows from (the dual of) \cite[V\S5 thm.~1]{Mac} that
  $\Fu[C]$ preserves colimits.

  Denote by $\Fu[I]$ the identity functor on $\C$. Since $\Fu[I]$ and
  $\Fu[C]$ preserve filtered colimits,
  \lemref{2-3-preserve-filtered-colimit} applies to the short exact
  sequence $0 \to \Fu[B] \to \Fu[I] \to \Fu[C] \to 0$ to show that
  $\Fu[B]$ preserves filtered colimits. From the short exact
  sequences,
  \begin{equation*}
    0 \lra \Fu[H] \lra \Fu[C] \lra \Shift{\Fu[B]} \lra 0 \qqand
    0 \lra \Fu[B] \lra \Fu[Z] \lra \Fu[H] \lra 0
  \end{equation*}
  we now conclude that $\Fu[H]$ and $\Fu[Z]$ preserve filtered
  colimits as well.
\end{proof}

\begin{prp}
  \label{prp:Hom-colim}
  For every $R$-complex $N$ the functor
  $\mapdef{\tp{\blank}{N}}{\C[\Rop]}{\C[\ZZ]}$ preserves colimits.
  For every bounded $R$-complex $P$ of finitely generated projective
  modules the functor $\mapdef{\Hom{P}{\blank}}{\C}{\C[\ZZ]}$
  preserves colimits.
\end{prp}

\enlargethispage*{2\baselineskip}
\begin{proof}
  The functor $\tp{\blank}{N}$ has a right adjoint, namely
  $\Hom[\ZZ]{N}{\blank}$, so it follows from (the dual of) \cite[V\S5
  thm.~1]{Mac} that $\tp{\blank}{N}$ preserves colimits.

  If $P$ is a bounded complex of finitely generated projective
  $R$-modules, then $\Hom{P}{R}$ is a bounded complex of finitely
  generated projective $\Rop$-modules. By \ref{bid} and \ref{hev}
  there are natural isomorphisms of functors from $\C$ to $\C[\ZZ]$,
  \begin{align*}
    \Hom{P}{\blank}
    &\is \Hom{\Hom[\Rop]{\Hom{P}{R}}{R}}{\blank} \\
    &\is \tp[\Rop]{\Hom{R}{\blank}}{\Hom{P}{R}} \\
    &\is \tp[\Rop]{\blank}{\Hom{P}{R}}\,,
  \end{align*}
  and the desired conclusion follows from the first assertion.
\end{proof}

\section{Finitely presented objects in the category of complexes} 
\label{sec:fp}

\noindent
Let $\Catfont{A}$ be an additive category. Following Crawley-Boevey
\cite{WCB94}, an object $A$ in $\Catfont{A}$ is called \emph{finitely
  presented} if the functor $\Catfont{A}(A,\blank)$ preserves filtered
colimits. The category $\Catfont{A}$ is called \emph{locally finitely
  presented} if the category of finitely presented objects in
$\Catfont{A}$ is skeletally small and if every object in $\Catfont{A}$
is a filtered colimit of finitely presented objects; see~\cite{WCB94}.

\begin{dfn}
  For an $R$-module $F$ and $\dinZ$ denote by $\Ct[\d]{F}$ the
  $R$-complex \smash{$0 \lra F \dxra{=} F \lra 0$} concentrated in
  degrees $\d$ and $\d-1$.
\end{dfn}

\begin{con}
  \label{con:fpcx}
  Let $M$ be an $R$-complex. For a homomorphism
  $\mapdef{\pi}{F}{M_\d}$ of $R$-modules, there is a morphism of
  $R$-complexes, $\mapdef{\widetilde{\pi}}{\Ct[\d]{F}}{M}$, given by
  \begin{equation*}
    \xymatrix{
      {} & 0 \ar[r] & 
      F \ar[d]^-{\pi} \ar[r]^-{\Id{F}} & 
      F \ar[d]^-{\dif[\d]{M}\pi} \ar[r] & 0 & {} \\
      \cdots \ar[r] & 
      M_{\d+1} \ar[r]^-{\dif[\d+1]{M}} & 
      M_{\d} \ar[r]^-{\dif[\d]{M}} & 
      M_{\d-1} \ar[r]^-{\dif[\d-1]{M}} & 
      M_{\d-2} \ar[r] & \cdots\xycomma[.]
    }
  \end{equation*}
\end{con}

\begin{prp}
  \label{prp:fp-complex}
  An $R$-complex $M$ is bounded and degreewise finitely presented if
  and only if there exists an exact sequence of $R$-complexes $L^1 \to
  L^0 \to M \to 0$ where $L^0$ and $L^1$ are bounded complexes of
  finitely generated free modules.
\end{prp}

\begin{proof}
  The ``if'' part is trivial. To show ``only if'', assume that $M$ is
  a bounded and degreewise finitely presented complex. Since $M$ is,
  in particular, degreewise finitely generated, we can for every
  $\dinZ$ choose a surjective homomorphism
  $\mapdef{\pi^\d}{F^\d}{M_\d}$ where $F^\d$ is finitely generated
  free, and such that $F^\d$ is zero if $M_\d$ is zero. Set $L^0 =
  \coprod_{\dinZ} \Ct[\d]{F^\d}$ and let $\mapdef{\pi}{L^0}{M}$ be the
  unique morphism whose composite, $\pi\e^\d$, with the embedding
  $\mapdef[\into]{\e^\d}{\Ct[\d]{F^\d}}{L^0}$ equals the morphism
  $\tilde{\pi}^\d$ from \conref[]{fpcx}. Evidently, $\pi$ is
  surjective and $L^0$ is a bounded complex of finitely generated free
  modules. Consider the kernel $M' = \Ker{\pi}$. Since $L^0$ is
  bounded, so is $M'$. Furthermore, as $M$ is degreewise finitely
  presented, $M'$ is degreewise finitely generated. Hence the argument
  above shows that there exists a surjective morphism $L^1 \to M'$
  where $L^1$ is a bounded complex of finitely generated free
  modules. The composite $L^1 \onto M' \into L^0$ now yields the
  left-hand morphism in an exact sequence $L^1 \to L^0 \to M \to 0$.
\end{proof}

It is a well-known fact that every module is isomorphic to a direct
limit of finitely presented modules. The following generalization to
complexes can be found in Garc{\'{\i}}a Rozas's \cite[lems.~4.1.1(ii)
and 5.1.1]{JGR99}.

\begin{ipg}
  \label{every-complex-is-colim-of-fp}
  Every $R$-complex is isomorphic to a direct limit of bounded and
  degreewise finitely presented $R$-complexes. \qed
\end{ipg}

The next theorem identifies the finitely presented objects in the
category $\C$, and combined with \pgref{every-complex-is-colim-of-fp}
it shows that this category is locally finitely presented; see
\corref{fp-objects} and \rmkref{C-is-locally-finitely-presented}.

\begin{thm}
  \label{thm:colimHom}
  For an $R$-complex $M$ the following conditions are equivalent.
  \begin{eqc}
  \item $M$ bounded and degreewise finitely presented.
  \item The functor $\Hom{M}{\blank}$ preserves filtered colimits.
  \item The functor $\Hom{M}{\blank}$ preserves direct limits.
  \item The functor $\C(M,\blank)$ preserves filtered colimits.
  \item The functor $\C(M,\blank)$ preserves direct limits.
  \end{eqc}
\end{thm}

\begin{proof}
  The implications \proofofimp[]{ii}{iii} and \proofofimp[]{iv}{v} are
  trivial. By \ref{Hom} there is for every $R$-complex $M$ an identity
  of functors from $\C$ to $\ZZ$-modules,
  \begin{equation*}
    \C(M,\blank) \deq \Cy[0]{\Hom{M}{\blank}}\;.
  \end{equation*}
  By \prpref{ZCBH} the functor $\Fu[Z]_0$ preserves filtered colimits,
  and hence the implications \proofofimp[]{ii}{iv} and
  \proofofimp[]{iii}{v} follow. It remains to show that the
  implications \proofofimp[]{i}{ii} and \proofofimp[]{v}{i} hold.

  \proofofimp{i}{ii} By \prpref{fp-complex} there is an exact sequence
  $L^1 \to L^0 \to M \to 0$ where $L^0$ and $L^1$ are bounded
  complexes of finitely generated free $R$-modules. Thus there is an
  exact sequence,
  \begin{equation*}
    0 \lra \Hom{M}{\blank} \lra \Hom{L^0}{\blank} \lra \Hom{L^1}{\blank}\,,
  \end{equation*}
  of functors from $\C$ to $\C[\ZZ]$. By \prpref{Hom-colim} the
  functors $\Hom{L^0}{\blank}$ and $\Hom{L^1}{\blank}$ preserve
  filtered colimits, and the conclusion follows from
  \lemref{2-3-preserve-filtered-colimit}.

  \proofofimp{v}{i} By \ref{every-complex-is-colim-of-fp} there is
  a direct system $\{\mapdef{\mu^{vu}}{M^u}{M^v}\}_{u \le v}$ of
  bounded and degreewise \fp\ $R$-complexes with $\colim M^u \is
  M$. By assumption, the canonical morphism
  \begin{equation*}
    \dmapdef{\a}{\colim \C(M,M^u)}{\C(M,\colim M^u) \dis \C(M,M)}
  \end{equation*}
  is an isomorphism. Write
  \begin{equation*}
    \dmapdef{\mu^u}{M^u}{\colim M^u \dis M} \qand
    \dmapdef{\lambda^u}{\C(M,M^u)}{\colim \C(M,M^u)}
  \end{equation*}
  for the canonical morphisms, and note that $\a\lambda^u =
  \C(M,\mu^u)$ holds for all $u$. Surjectivity of $\a$ yields an
  element $\chi \in \colim \C(M,M^u)$ with $\a(\chi) = \Id{M}$.  By
  \ref{colim} one has $\chi=\lambda^u(\psi^u)$ for some $\psi^u \in
  \C(M,M^u)$. Hence, there are equalities $\mu^u\psi^u =
  \C(M,\mu^u)(\psi^u) = \a\lambda^u(\psi^u) = \a(\chi) =
  \Id{M}$. Thus, $M$ is a direct summand of $M^u$, and since $M^u$ is
  bounded and degreewise \fp, so is $M$.
\end{proof}

The equivalences above of \eqclbl{ii} and \eqclbl{iii} and of \eqclbl{iv}
and \eqclbl{v} also follow from general principles; see \cite[cor.~to
thm.~1.5]{AdamekRosicky}.

\begin{cor}
  \label{cor:fp-objects}
  The category $\C$ is locally finitely presented, and the finitely
  presented objects in $\C$ are exactly the bounded and degreewise
  finitely presented $R$-complexes.
\end{cor}

\begin{proof}
  By the equivalence of \eqclbl{i} and \eqclbl{iv} in
  \thmref{colimHom}, the finitely presented objects in the category
  $\C$ are exactly the bounded and degreewise finitely presented
  $R$-complexes. Evidently, the category of such complexes is
  skeletally small. By \ref{every-complex-is-colim-of-fp} every
  object in $\C$ is a filtered colimit (even a direct limit) of
  finitely presented objects.
\end{proof}

\begin{rmk}
  \label{rmk:C-is-locally-finitely-presented}
  The fact that $\C$ is locally finitely presented also follows from
  \cite[Satz 1.5]{SBr70} and \cite[(2.4)]{WCB94}; indeed, $\C$ is a
  Grothendieck category and \mbox{$\setof{\Ct[u]{R}\!}{\!u\in\ZZ}$} is
  a generating set of finitely presented objects.
\end{rmk}

\section{Proof of the main theorem}
\label{sec:main}

\noindent
The notion of semi-flatness, and the related notions of
semi-projectivity and semi-freeness, originate in the treatise
\cite{dga} by Avramov, Foxby, and Halperin.

A graded $R$-module $L$ is called \emph{graded-free} if it has a
\emph{graded basis}, that is, a basis consisting of homogeneous
elements. It is easily seen that $L$ is graded-free if and only if
every component $L_\d$ is a free $R$-module.

\begin{ipg}
  \label{semi-free}
  An $R$-complex $L$ is called %
  \emph{semi-free} if the underlying graded $R$-module has a graded
  basis $E$ that can be written as a disjoint union \smash{$E =
    \biguplus_{n\ge 0} E^n$} such that one has \smash{$E^0 \subseteq
    \Cy{L}$} and \smash{$\dif{L}(E^n) \subseteq
    \submod{\bigcup_{i=0}^{n-1} E^i}$} for every $n \ge 1$. Such a
  basis is called a \emph{semi-basis} of $L$.
\end{ipg}

\begin{exa}
  \label{exa:semi-free}
  A bounded below complex of free modules is semi-free.
\end{exa}

\begin{ipg}
  \label{semifree-res}
  Every $R$-complex $M$ has a \emph{semi-free resolution}, that is, a
  quasi-isomorphism of $R$-complexes $\mapdef{\pi}{L}{M}$ where $L$ is
  semi-free.  Moreover, $\pi$ can be chosen surjective and with
  $L_\d=0$ for all $\d < \inf\setof{n\in\ZZ}{M_n\ne
    0}$. See~\cite[thm.~2.2]{dga}.
\end{ipg}

\begin{ipg}
  \label{semi-projective}
  For an $R$-complex $P$ the following conditions are equivalent.
  \begin{eqc}
  \item The functor $\Hom{P}{\blank}$ is exact and preserves
    quasi-isomorphisms.
  \item For every morphism $\mapdef{\a}{P}{N}$ and for every
    surjective quasi-isomorphism $\mapdef{\b}{M}{N}$ there exists a
    morphism $\mapdef{\g}{P}{M}$ such that $\a = \b\g$ holds.
  \item $P$ is a complex of projective $R$-modules, and the functor
    $\Hom{P}{\blank}$ preserves acyclicity.
  \end{eqc}
  A complex that satisfies these equivalent conditions is called
  \emph{semi-projective}; see~\cite[thm.~3.5]{dga}.
\end{ipg}

\begin{exa}
  \label{exa:semi-projective}
  A bounded below complex of projective modules is semi-projective.

  By \thmcite[3.5]{dga} a semi-free complex is semi-projective.
\end{exa}

\begin{ipg}
  \label{semi-flat}
  For an $R$-complex $F$ the following conditions are equivalent.
  \begin{eqc}
  \item The functor $\tp{\blank}{F}$ is exact and preserves
    quasi-isomorphisms.
  \item $F$ is a complex of flat $R$-modules and the functor
    $\tp{\blank}{F}$ preserves acyclicity.
  \end{eqc}
  A complex that satisfies these equivalent conditions is called
  \emph{semi-flat}; see~\cite[thm.~6.5]{dga}.
\end{ipg}

\begin{exa}
  \label{exa:semi-flat}
  A bounded below complex of flat modules is semi-flat.

  By \lemcite[7.1]{dga} a semi-projective complex is semi-flat.
\end{exa}

As noted in \pgref{bid}, the biduality morphism $\bid{P}{}$ is an
isomorphism for every complex $P$ of finitely generated projective
modules. For the proof of \thmref{GL} we need an explicit description
of the inverse.

\begin{lem}
  \label{lem:bid-inv}
  For a complex $P$ of \fg\ projective $R$-modules, the inverse of the
  isomorphism $\bid{\Hom{P}{R}}{}$ is $\Hom{\bid{P}{}}{R}$.
\end{lem}

\begin{proof} 
  As $P$ is a complex of \fg\ projective $R$-modules, $\bid{P}{}$ and
  hence $\Hom{\bid{P}{}}{R}$ are isomorphisms by \ref{bid}.  For $\f$
  in $\Hom{P}{R}$ and $x$ in $P$ one has
  \begin{align*}
    (\Hom{\bid{P}{}}{R}\bid{\Hom{P}{R}}{})(\psi)(x)
    &= (\bid{\Hom{P}{R}}{}(\psi)\bid{P}{})(x)\\
    &= \bid{\Hom{P}{R}}{}(\psi)(\bid{P}{}(x))\\
    &= \bid{P}{}(x)(\psi)\\
    &= \psi(x) \,,
  \end{align*}
  so \smash{$\Hom{\bid{P}{}}{R}\bid{\Hom{P}{R}}{}$} is the identity on
  $\Hom{P}{R}$.
\end{proof}

Condition \eqclbl{iii} in \thmref{GL} asserts the existence of a
certain pure epimorphism in $\C$. In the proof below, we use that the
equivalence of conditions \eqclbl{ii}--\eqclbl{v} has been
established elsewhere and we do not directly address
\eqclbl{iii}. However, in the next section we study the relationship
between purity and semi-flatness; in particular, we recall the
definition of a pure epimorphism from \cite[\S3]{WCB94} in the first
paragraph of \secref{purity}.

\begin{proof}[Proof of \thmref{GL}]
  By \corref{fp-objects} the category $\C$ is locally finitely
  presented and its finitely presented objects are exactly the bounded
  and degreewise finitely presented complexes. It now follows from
  \cite[(4.1)]{WCB94} that \eqclbl{ii}, \eqclbl{iii}, and \eqclbl{iv}
  are equivalent. Furthermore, \cite[thm.~1.5]{AdamekRosicky} shows
  that \eqclbl{iv} and \eqclbl{v} are equivalent.  The remaining
  implications \proofofimp[]{i}{ii} and \proofofimp[]{v}{i} are proved
  below.

  \proofofimp{i}{ii} Let $\mapdef{\f}{N}{F}$ be a morphism of
  $R$-complexes where $N$ is bounded and degreewise finitely
  presented.  By \prpref{fp-complex} there is an exact sequence,
  \begin{equation*}
    L^1 \xra{\psi^1} L^0 \xra{\psi^0} N \lra 0 \,,
  \end{equation*}
  of $R$-complexes, where $L^0$ and $L^1$ are bounded complexes of
  finitely generated free modules. Consider the exact sequence of
  $\Rop$-complexes,
  \begin{equation}
    \label{eq:1c}
    0 \lra K \xra{\iota} \Hom{L^0}{R} \dxra{\Hom[]{\psi^1}{R}} \Hom{L^1}{R}\,,
  \end{equation}
  where $K$ is the kernel of $\Hom{\psi^1}{R}$ and $\iota$ is the
  embedding.  The functor $\Cy[0]{\blank}$ is left exact, and the
  functor $\tp{\blank}{F}$ is exact by definition, so it follows that
  the functor $\Cy[0]{\tp{\blank}{F}}$ leaves the sequence \eqref{1c}
  exact.  As $L^0$ is bounded, so is $K$; set
  $u=\inf\setof{n\in\ZZ}{K_n\ne 0}$.  By \ref{semifree-res} there is
  an exact sequence,
  \begin{equation}
    \label{eq:2}
    P \xra{\pi} K \lra 0\,,
  \end{equation}
  where $\pi$ is a quasi-isomorphism and $P$ is a semi-free
  $\Rop$-complex with $P_\d=0$ for all $\d<u$.  As $F$ is semi-flat,
  $\tp{\pi}{F}$ is a surjective quasi-isomorphism
  by~\ref{semi-flat}. A simple diagram chase shows that every
  surjective quasi-isomorphism is surjective on cycles, so the functor
  $\Cy[0]{\tp{\blank}{F}}$ leaves the sequence \eqref{2}
  exact. Consequently, there is an exact sequence,
  \begin{equation*}
    \Cy[0]{\tp{P\mspace{-2mu}}{\mspace{-2mu}F}}
    \dxra{\mspace{-5mu}\tp[]{(\iota\pi)}{F}\mspace{-5mu}} 
    \Cy[0]{\tp{\Hom{L^0}{R}\mspace{-2mu}}{\mspace{-2mu}F}} 
    \dxra{\mspace{-5mu}\tp[]{\Hom[]{\psi^1}{R}}{F}\mspace{-5mu}} 
    \Cy[0]{\tp{\Hom{L^1}{R}\mspace{-2mu}}{\mspace{-2mu}F}}\;.
  \end{equation*}

  For every $R$-complex $M$, denote by $\xi^M$ the composite morphism
  \begin{equation*}
    \tp{\Hom{M}{R}}{F} \dxra{\tev{MRF}}
    \Hom{M}{\tp{R}{F}} \ira
    \Hom{M}{F}\,,
  \end{equation*}
  where $\tev{MRF}$ is the tensor evaluation morphism \ref{tev} and
  the isomorphism is induced by the canonical one $\tp{R}{F} \is
  F$. The morphism $\xi^M$ is natural in $M$, and by \ref{tev} it is
  an isomorphism if $M$ is a bounded complex of finitely generated
  projective modules.  The exact sequence above now yields another
  exact sequence,
  \begin{equation}
    \label{eq:3}
    \Cy[0]{\tp{P}{F}} \dxra{\xi^{L^0} \circ\,(\tp[]{(\iota\pi)}{F})} 
    \Cy[0]{\Hom{L^0}{F}} \dxra{\Hom[]{\psi^1}{F}} 
    \Cy[0]{\Hom{L^1}{F}}\;.
  \end{equation}

  As $\mapdef{\f\psi^0}{L^0}{F}$ is a morphism, it is an element in
  $\Cy[0]{\Hom{L^0}{F}}$; see~\ref{Hom}.  Since one has
  $\Hom{\psi^1}{F}(\f\psi^0) = \f\psi^0\psi^1 = 0$, exactness of
  \eqref{3} yields an element $x$ in $\Cy[0]{\tp{P}{F}}$ with
  \begin{equation}
    \label{eq:4}
    (\xi^{L^0} \circ (\tp{(\iota\pi)}{F}))(x) \deq \f\psi^0\;.
  \end{equation}
  The graded module underlying $P$ has a graded basis $E$, and $x$ has
  the form $x=\sum_{i=1}^n e_i \otimes f_i$ with $e_i \in E$ and $f_i
  \in F$. Set $w=\max\{\dgr{e_1},\ldots,\dgr{e_n}\}$; as one has
  $P_\d=0$ for all $\d<u$, each basis element $e_i$ satisfies $u \le
  \dgr{e_i} \le w$. For $\d\in\ZZ$ set $E_\d = \setof{e\in E}{\dgr{e}
    = \d}$. Next we define a bounded subcomplex $P'$ of $P$ such that
  each module $P'_\d$ is finitely generated and free. For $\d \notin
  \{u,\ldots,w\}$ set $P'_\d=0$; for $\d \in \{u,\ldots,w\}$ the
  modules $P'_\d$ are constructed inductively.  Let $P'_w$ be the
  finitely generated free submodule of $P_w$ generated by the set
  $E'_w = \{e_1,\ldots,e_n\} \cap E_w $.  For $\d \le w$ assume that a
  finitely generated free submodule $P'_\d$ of $P_\d$ with finite
  basis $E'_\d$ has been constructed.  As the subset $B'_{\d-1} \deq
  \setof{\dif{P}(e)}{e \in E'_\d}$ of $P_{\d-1}$ is finite, there is a
  finite subset $G'_{\d-1}$ of $E_{\d-1}$ with $B'_{\d-1} \subseteq
  \submod[\Rop]{G'_{\d-1}}$.  Now let $P'_{\d-1}$ be the submodule of
  $P_{\d-1}$ generated by the following finite set of basis elements,
  \begin{equation*}
    E'_{\d-1} \deq G'_{\d-1} \cup (\{e_1,\ldots,e_n\} \cap E_{\d-1})\;.    
  \end{equation*}
  By construction, one has $\dif{P}(P'_\d) \subseteq P'_{\d-1}$ for
  all $\dinZ$, so $P'$ is a subcomplex of $P$.  The construction shows
  that $x=\sum_{i=1}^n e_i \otimes f_i$ belongs to $\tp{P'}{F}$.  As
  $F$ is a complex of flat $R$-modules, $\tp{P'}{F}$ is a subcomplex
  of $\tp{P}{F}$, and as the element $x$ is in $\Cy[0]{\tp{P}{F}}$ it
  also belongs to $\Cy[0]{\tp{P'}{F}}$.

  Set $L=\Hom[\Rop]{P'}{R}$. As $P'$ is a bounded complex of finitely
  generated free $\Rop$-modules, $L$ is a bounded complex of finitely
  generated free $R$-modules. Let $\mapdef[\into]{\e}{P'}{P}$ be the
  embedding and let $\mapdef{\kappa'}{L^0}{L}$ be the composite
  morphism
  \begin{equation*}
    L^0 \dxra{\bid{L^0}{}} 
    \Hom[\Rop]{\Hom{L^0}{R}}{R} \dxra{\Hom[]{\iota\pi\e}{R}}
    \Hom[\Rop]{P'}{R} \deq L\;.
  \end{equation*}
  In the commutative diagram
  \begin{equation*}
    \xymatrix@C=6.5pc{
      P' \ar[d]^-{\bid{P'}{}}_-{\is} \ar[r]^-{\iota\pi\e} 
      & \Hom{L^0}{R} \ar[d]^-{\bid{\Hom[]{L^0}{R}}{}}_-{\is}
      \\
      \Hom{\Hom[\Rop]{P'}{R}}{R} 
      \ar[r]^-{\Hom[]{\Hom[]{\iota\pi\e}{R}}{R}} & 
      \Hom{\Hom[\Rop]{\Hom{L^0}{R}}{R}}{R}
    }
  \end{equation*}
  the vertical morphisms are isomorphisms by \ref{bid}, and
  \smash{$\bid{\Hom[]{L^0}{R}}{}$} is by \lemref{bid-inv} the inverse
  of \smash{$\Hom{\bid{L^0}{}}{R}$}. One now has
  \begin{equation}
    \label{eq:5}
    \Hom{\kappa'}{R}\bid{P'}{} \deq \iota\pi\e\;.   
  \end{equation}
  It follows that there are equalities,
  \begin{equation*}
    \Hom{\kappa'\psi^1}{R}\bid{P'}{} \deq
    \Hom{\psi^1}{R}\iota\pi\e \deq 0\pi\e \deq 0\,,
  \end{equation*}
  and since $\bid{P'}{}$ is an isomorphism, the morphism
  $\Hom{\kappa'\psi^1}{R}$ is zero. In particular,
  $\Hom[\Rop]{\Hom{\kappa'\psi^1}{R}}{R}$ is zero, and hence the
  commutative diagram
  \begin{equation*}
    \xymatrix@C=7pc{
      L^1 \ar[d]^-{\bid{L^1}{}}_-{\is} \ar[r]^-{\kappa'\psi^1} 
      & L \ar[d]^-{\bid{L}{}}_-{\is}
      \\
      \Hom[\Rop]{\Hom{L^1}{R}}{R} 
      \ar[r]^-{\Hom[]{\Hom[]{\kappa'\psi^1}{R}}{R}} & 
      \Hom[\Rop]{\Hom{L}{R}}{R} 
    }
  \end{equation*}
  shows that \smash{$\kappa'\psi^1=0$} holds. Again the vertical
  morphisms are isomorphisms by \ref{bid}. Since \smash{$\kappa'$}
  vanishes on \smash{$\Im{\psi^1} = \Ker{\psi^0}$} there is a unique
  morphism $\mapdef{\kappa}{N}{L}$ with $\kappa\psi^0 = \kappa'$.
  Finally, consider the diagram,
  \begin{equation}
    \label{eq:7}
    \begin{gathered}
      \xymatrix@C=2.8pc{ \tp{P'}{F} \ar[d]^-{\tp[]{\e}{F}}
        \ar[r]^-{\tp[]{\bid{P'}{}}{F}} & \tp{\Hom{L}{R}}{F}
        \ar[d]^-{\tp[]{\Hom[]{\kappa'}{R}}{F}} \ar[r]^-{\xi^L} &
        \Hom{L}{F} \ar[d]^-{\Hom{\kappa'}{F}} \\
        \tp{P}{F} \ar[r]^-{\tp[]{(\iota\pi)}{F}} &
        \tp{\Hom{L^0}{R}}{F} \ar[r]^-{\xi^{L^0}} &
        \Hom{L^0}{F}\xycomma }
    \end{gathered}
  \end{equation}
  where the left-hand square is commutative by \eqref{5} and the
  right-hand square is commutative by naturality of $\xi$.  Set
  \begin{equation*}
    \lambda \deq (\xi^L \circ (\tp[]{\bid{P'}{}}{F}))(x)\;;
  \end{equation*}
  it is an element in $\Hom{L}{F}$, and as $x$ belongs to
  $\Cy[0]{\tp{P'}{F}}$, also $\l$ is a cycle;
  i.e.~$\mapdef{\lambda}{L}{F}$ is a morphism. From \eqref{7}, from
  the definition of $\lambda$, and from \eqref{4} one gets
  $\lambda\kappa' = \f\psi^0$.  The identity $\kappa' = \kappa\psi^0$
  and surjectivity of $\psi^0$ now yield $\lambda\kappa=\f$.

  \proofofimp{v}{i} Every bounded complex of finitely generated free
  $R$-modules is semi-flat, see~\exaref{semi-flat}, and as mentioned
  in the introduction a direct limit of semi-flat complexes is
  semi-flat. A proof of this fact can be found in \prpcite[6.9]{dga};
  for completeness we include the argument.  Let \mbox{$\{F^u \to
    F^v\}_{u \le v}$} be a direct system of semi-flat $R$-complexes
  and set $F=\colim F^u$. By \prpref{Hom-colim} there is a natural
  isomorphism of functors, $\tp{\blank}{F} \is
  \colim(\tp{\blank}{F^u})$. By assumption, each functor
  $\tp{\blank}{F^u}$ is exact and preserves acyclicity. Since direct
  limits in $\C[\ZZ]$ are exact, see~\ref{colim}, and since the
  homology functor preserves direct limits, see~\prpref{ZCBH}, it
  follows that the functor $\tp{\blank}{F}$ is exact and preserves
  acyclicity; that is, $F$ is semi-flat.
\end{proof}

\section{Purity}
\label{sec:purity}

\noindent
Let $\Catfont{A}$ be a locally finitely presented category.  Following
\cite[\S3]{WCB94} a short exact sequence $0 \to M' \to M \to M'' \to
0$ in $\Catfont{A}$ is called a \emph{pure} if
\begin{equation*}
  0 \lra \Catfont{A}(A,M') \lra \Catfont{A}(A,M) \lra
  \Catfont{A}(A,M'') \lra 0 
\end{equation*}
is exact for every finitely presented object $A$ in $\Catfont{A}$. In
this case, the morphism $M' \to M$ is called a \emph{pure
  monomorphism} and $M \to M''$ is called \emph{pure epimorphism}.

In view of \corref{fp-objects}, a morphism $\mapdef{\a}{X}{Y}$ in $\C$
is a pure epimorphism if and only if for every morphism
$\mapdef{\f}{N}{Y}$ with $N$ bounded and degreewise \fp\ there exists
a morphism $\mapdef{\b}{N}{X}$ with $\f = \a\b$.

Semi-flat complexes have the following two-out-of-three property; see
the proof of \prpcite[6.7]{dga}.

\begin{ipg}
  \label{2-3-semiflat}
  Let $0 \to F' \to F \to F'' \to 0$ be an exact sequence of
  $R$-complexes.  If $F''$ is semi-flat, then $F'$ is semi-flat if and
  only if $F$ is semi-flat.
\end{ipg}

The next result supplements \ref{2-3-semiflat}; it shows that the
class of semi-flat complexes is closed under pure subcomplexes and
pure quotient complexes.

\begin{prp}
  \label{prp:pure}
  Let \mbox{$0 \to F' \to F \to F'' \to 0$} be a pure exact sequence
  of $R$-complexes. If the complex $F$ is semi-flat, then $F'$ and
  $F''$ are semi-flat.
\end{prp}

\begin{proof}
  Assume that $F$ is semi-flat and denote the given morphism from $F$
  to $F''$ by $\a$.  Let $\mapdef{\f}{N}{F''}$ be a morphism where $N$
  is a bounded and degreewise finitely presented $R$-complex.  Since
  $\a$ is a pure epimorphism one has $\f = \a\b$ for some morphism
  $\mapdef{\b}{N}{F}$. Since $F$ is semi-flat, the morphism $\b$, and
  hence also $\f$, factors through a bounded complex of finitely
  generated free $R$-modules. Thus $F''$ is semi-flat by
  \thmref{GL}. It now follows from \ref{2-3-semiflat} that $F'$ is
  semi-flat as well.
\end{proof}

Every surjective homomorphism $M \to F$ of $R$-modules with $F$ flat
is a pure epimorphism; see \corcite[1.3]{DLz69}. The next example
shows that a surjective morphism $M \to F$ of $R$-complexes with $F$
semi-flat need not be a pure epimorphism.

\begin{exa}
  Consider the $\ZZ$-complexes, $\Ct[0]{\ZZ}$ and $\ZZ$. As a
  $\ZZ$-complex, $\ZZ$ is semi-flat by \exaref{semi-flat}. The
  surjective morphism $\mapdef{\pi}{\Ct[0]{\ZZ}}{\ZZ}$, given by the
  diagram
  \begin{equation*}
    \xymatrix{
      0 \ar[r] & \ZZ \ar[d]^-{\Id{\ZZ}} \ar[r]^-{\Id{\ZZ}} 
      & \ZZ \ar[d] \ar[r] & 0 \\
      0 \ar[r] & \ZZ \ar[r] & 0 \ar[r] & 0\xycomma 
    }
  \end{equation*}
  is not a pure epimorphism. Indeed, the complex $\ZZ$ is finitely
  presented but the identity morphism $\ZZ \to \ZZ$ does not factor
  through $\pi$; in other words, $\pi$ is not a split surjection.
\end{exa}

What can be salvaged is captured in the next proposition.

\begin{prp}
  \label{prp:pure-semiflat}
  For an $R$-complex $F$ the following conditions are equivalent.
  \begin{eqc}
  \item $F$ is semi-flat.
  \item Every surjective quasi-isomorphism $M \to F$ a is pure
    epimorphism.
  \item There exists a semi-free complex $L$ and a quasi-isomorphism
    $L \to F$ which is also a pure epimorphism.
  \end{eqc}
\end{prp}

\begin{proof}
  \proofofimp{i}{ii} Let $\mapdef{\a}{M}{F}$ be a surjective
  quasi-isomorphism and $\mapdef{\f}{N}{F}$ be a morphism with $N$
  bounded and degreewise \fp. Since $F$ is semi-flat there exists by
  \thmref{GL} a bounded complex $L$ of \fg\ free $R$-modules and
  morphisms $\mapdef{\k}{N}{L}$ and $\mapdef{\l}{L}{F}$ with $\f =
  \l\k$. As $L$ is semi-projective, see \exaref[Examples~]{semi-free}
  and \exaref[]{semi-projective}, there exists by
  \ref{semi-projective} a morphism $\mapdef{\g}{L}{M}$ with $\l=\a\g$,
  so with $\b=\g\k$ one has $\f = \a\b$.

  \proofofimp{ii}{iii} Immediate from \ref{semifree-res}.

  \proofofimp{iii}{i} Follows from \prpref{pure} as a semi-free
  complex is semi-flat.
\end{proof}

\section{Semi-flat complexes of projective modules}
\label{sec:8}

\noindent
A semi-free complex is semi-projective, see~\exaref{semi-projective},
but a semi-projective complex of free modules need not be
semi-free. Indeed, the $\ZZ/6\ZZ$-complex,
\begin{equation*}
  \cdots \lra \ZZ/6\ZZ \xra{2} \ZZ/6\ZZ \xra{3} \ZZ/6\ZZ \xra{2} 
  \ZZ/6\ZZ \xra{3} \ZZ/6\ZZ \lra \cdots
\end{equation*}
serves as a counterexample; see \exacite[7.10]{dga}. It turns out that
a semi-flat complex of projective modules is, in fact,
semi-projective. As Murfet notes in his thesis
\corcite[5.14]{DMf-phd}, this follows from work of Neeman \cite{ANm08}
on the homotopy category of flat modules. The purpose of this section
is to provide an alternative proof of this fact.

\begin{ipg}
  \label{contractible}
  For an $R$-complex $C$, the following conditions are equivalent.
  \begin{eqc}
  \item The identity $\Id{C}$ is null-homotopic; that is, it is a
    boundary in $\Hom{C}{C}$.
  \item The exists a degree $1$ homomorphism $\mapdef{\s}{C}{C}$ with
    $\dif{C} = \dif{C}\s\dif{C}$.
  \item There exists a graded $R$-module $B$ with $\Cone{\Id{B}} \is
    C$.
  \end{eqc}
  A complex that satisfies these conditions is called
  \emph{contractible}; see~\seccite[1.4]{Wei}.

  If $C$ is contractible, then so are all complexes $\Hom{C}{X}$,
  $\Hom{X}{C}$, and $\tp{Y}{C}$. Every contractible complex is
  acyclic.
\end{ipg}

\begin{lem}
  \label{lem:contractible}
  Let $N$ be a bounded and degreewise finitely generated $R$-complex,
  and let $C$ be a contractible complex of projective
  $R$-modules. Every morphism $N\to C$ factors as $N \to L \to C$,
  where $L$ is a bounded and contractible complex of finitely
  generated free $R$-modules.
\end{lem}

\begin{proof}
  There is a graded $R$-module $P$ with $C \is \Cone{\Id{P}} =
  \coprod_{\d\in\ZZ} \Ct[\d+1]{P_\d}$. In particular, $C$ is a
  coproduct of bounded contractible complexes of projective
  $R$-modules. For each module $P_\d$ there is a complementary module
  $Q_\d$ and a set $E_\d$ such that there is an isomorphism
  $P_\d\oplus Q_\d \is R^{(E_v)}$. Set
  \begin{equation*}
    L' = C \oplus \big(\coprod_{\d\in\ZZ} \Ct[\d+1]{P_\d}\big) 
    \is \coprod_{\d\in\ZZ} (\Ct[\d+1]{R})^{(E_v)}\,.
  \end{equation*}
  A morphism $\mapdef{\a}{N}{C}$ factors through $L'$, and since $N$
  is bounded and degreewise finitely generated, it factors through a
  finite coproduct $L = \finsum_{i=1}^n \Ct[\d_i+1]{R}$. Evidently,
  this is a bounded and contractible complex of finitely generated
  free $R$-modules.
\end{proof}

The next result shows that the complexes characterized in
\thmcite[2.4]{EEEJGR98}, in \thmcite[4.1.3]{JGR99}, and in \cite[fact
2.14]{ANm08} are precisely the acyclic semi-flat complexes. In
\cite{dga} such complexes are called \emph{categorically flat}, in
\cite{JGR99} they are called \emph{flat}, and in \cite{DMfSSl11} they
are called \emph{pure acyclic}.

\begin{thm}
  \label{thm:contractible}
  For an $R$-complex $F$ the following conditions are equivalent.
  \begin{eqc}
  \item $F$ is semi-flat and acyclic.
  \item $F$ is a filtered colimit of bounded and contractible
    complexes of finitely generated free $R$-modules.
  \item $F$ is a direct limit of bounded and contractible complexes of
    finitely generated free $R$-modules.
  \item $F$ is acyclic and $\Bo{F}$ is a complex of flat $R$-modules.
  \end{eqc}
\end{thm}

\begin{proof}
  \proofofimp{i}{ii} By \thmref{colimHom} and \cite[(4.1)]{WCB94} it
  is sufficient to prove that every morphism $\mapdef{\f}{N}{F}$ with
  $N$ bounded and degreewise finitely presented factors through a
  bounded and contractible complex of finitely generated free
  $R$-modules. Fix such a morphism $\f$.  Let
  $\mapdef[\qra]{\pi}{P}{F}$ be a surjective semi-free resolution;
  cf.~\pgref{semifree-res}. As $P$ is acyclic and semi-projective, the
  complex $\Hom{P}{P}$ is acyclic; in particular the morphism $\Id{P}$
  is null-homotopic so $P$ is contractible. The morphism $\pi$ is by
  \prpref{pure-semiflat} a pure epimorphism, so $\f$ factors through
  $P$ and hence, by \lemref{contractible}, through a bounded and
  contractible complex $L$ of finitely generated free $R$-modules.

  \proofofimp{ii}{iii} This follows from
  \cite[thm.~1.5]{AdamekRosicky}.

  \proofofimp{iii}{iv} A direct limit of contractible (acyclic)
  complexes is acyclic, so $F$ is acyclic. In a contractible complex
  $L$ of free $R$-modules, the subcomplex $\Bo{L}$ consists of
  projective $R$-modules. The functor $\Bo{\blank}$ preserves direct
  limits by \prpref{ZCBH}, and a direct limit of projective modules is
  a flat module, so $\Bo{F}$ is a complex of flat $R$-modules.

  \proofofimp{iv}{i} Each sequence $0 \to \Bo[v]{F} \to F_v \to
  \Bo[v-1]{F} \to 0$ is exact, so each module $F_v$ is flat; that is,
  $F$ is an acyclic complex of flat $R$-modules. For an acyclic
  $\Rop$-complex $M$ (actually for any $\Rop$-complex) it follows from
  the K\"unneth formula \thmcite[3.6.3]{Wei} that $\tp{M}{F}$ is
  acyclic. Thus, $F$ is semi-flat by \pgref{semi-flat}.
\end{proof}

In the terminology of \cite{dga} the equivalence of \eqclbl{i} and
\eqclbl{iii} above says that a complex is categorically flat if and
only if it is a direct limit of categorically projective complexes of
\fg\ free modules.

\begin{cor}
  Let $\mapdef{\a}{F}{F'}$ be a quasi-isomorphism between semi-flat
  $R$-com\-plexes. For every bounded and degreewise finitely presented
  $R$-complex $N$ the morphism
  $\mapdef{\Hom{N}{\a}}{\Hom{N}{F}}{\Hom{N}{F'}}$ is a
  quasi-isomorphism.
\end{cor}

\begin{proof}
  By \ref{cone} and \ref{2-3-semiflat} the complex $\Cone{\a}$ is
  acyclic and semi-flat. The functor $\Hom{N}{\blank}$ preserves
  filtered colimits by \thmref{colimHom} and maps contractible
  complexes to contractible complexes. Now it follows from
  \thmref{contractible} and \prpref{ZCBH} that $\Hom{N}{\Cone{\a}}$ is
  acyclic. Since the functor $\Hom{N}{\blank}$ commutes with mapping
  cone, see \ref{Hom}, it follows that the complex
  $\Cone{\Hom{N}{\a}}$ is acyclic, and thus $\Hom{N}{\a}$ is a
  quasi-isomorphism by \ref{cone}.
\end{proof}

The next corollary can be proved similarly; a different proof is given
in \cite[6.4]{dga}.

\begin{cor}
  Let $\mapdef{\a}{F}{F'}$ be a quasi-isomorphism between semi-flat
  $R$-com\-plexes. For every $\Rop$-complex $M$ the morphism
  $\mapdef{\tp{M}{\a}}{\tp{M}{F}}{\tp{M}{F'}}$ is a
  quasi-isomorphism.\qed
\end{cor}

The next result was proved by Neeman \cite[rmk.~2.15 and
thm.~8.6]{ANm08} in 2008. He notes, ``I do not know an elementary
proof, a proof which avoids homotopy categories''.  We show that it
follows from a theorem of Benson and Goodearl
\cite[thm.~2.5]{DJBKRG00}\footnote{ Benson and Goodearl's proof uses
  only classical results from homological algebra.} from 2000, which
asserts that if $0 \to F \to P \to F \to 0$ is a short exact sequence
of $R$-modules with $F$ flat and $P$ projective, then $F$ is
projective as well. Notice that if $R$ has finite finitistic
projective dimension, then this assertion follows from Jensen's
\prpcite[6]{CUJ70}, and if $R$ has cardinality $\leqslant \aleph_n$ for some
$n\in\NN$, then it follows from a theorem of Gruson and Jensen
\thmcite[7.10]{LGrCUJ81}.

\begin{prp}
  \label{prp:Neemans-result}
  If $P$ is an acyclic complex of projective $R$-modules such that the
  subcomplex $\Bo{P}$ consists of flat $R$-modules, then $P$ is
  contractible.
\end{prp}

\begin{proof}
  For each $\dinZ$ the sequence $0 \to \Bo[\d]{P} \to P_\d \to
  \Bo[\d-1]{P} \to 0$ is exact. The coproduct of all these exact
  sequences yields the exact sequence
  \begin{equation*}
    0 \lra \dcoprod_{\dinZ}\Bo[\d]{P} \lra
    \dcoprod_{\dinZ}P_\d \lra
    \dcoprod_{\dinZ} \Bo[\d]{P} \lra 0\;.
  \end{equation*}
  By assumption, the module $\coprod_{\dinZ}\Bo[\d]{P}$ is flat and
  $\coprod_{\dinZ}P_\d$ is projective, so it follows from
  \cite[thm.~2.5]{DJBKRG00} that $\coprod_{\dinZ}\Bo[\d]{P}$ is
  projective. Consequently, every module $\Bo[\d]{P}$ is projective,
  and therefore $P$ is contractible.
\end{proof}

Semi-projective complexes, just like semi-flat complexes, have a
two-out-of-three property; see the proof of \prpcite[3.7]{dga}.

\begin{ipg}
  \label{2-3-semiproj}
  Let $0 \to P' \to P \to P'' \to 0$ be an exact sequence of
  $R$-complexes.  If $P''$ is semi-projective, then $P'$ is
  semi-projective if and only if $P$ is semi-projective.
\end{ipg}

\begin{thm}
  \label{thm:sfp}
  A semi-flat complex of projective $R$-modules is semi-projective.
\end{thm}

\begin{proof}
  Let $\qisdef{\pi}{L}{F}$ be a semi-free resolution, see
  \ref{semifree-res}, and consider the mapping cone sequence $0 \to F
  \to \Cone{\pi} \to \Shift{L} \to 0$; see~\ref{cone}. Since
  $\Shift{L}$ is semi-projective, see \exaref{semi-projective}, it
  suffices by \ref{2-3-semiproj} to argue that $\Cone{\pi}$ is
  semi-projective. As the complexes $F$ and $\Shift{L}$ are semi-flat
  and consist of projective modules, it follows from
  \ref{2-3-semiflat} and \ref{cone} that $\Cone{\pi}$ is an acyclic
  semi-flat complex of projective modules.  Thus \thmref{contractible}
  and \prpref{Neemans-result} apply to show that $\Cone{\pi}$ is
  contractible. It remains to note that every contractible complex of
  projective modules is semi-projective; this is immediate from
  \ref{semi-projective}.
\end{proof}

\section*{Acknowledgments}

\noindent
We thank Jan {\v{S}}{\v{t}}ov{\'{\i}}{\v{c}}ek for pointing us to
Ad{\'a}mek and Rosick{\'y}'s book. We also thank the anonymous referee
for helpful comments on both content and exposition.

\bibliographystyle{amsplain} 

\def\cprime{$'$}
  \providecommand{\arxiv}[2][AC]{\mbox{\href{http://arxiv.org/abs/#2}{\sf
  arXiv:#2 [math.#1]}}}
  \providecommand{\oldarxiv}[2][AC]{\mbox{\href{http://arxiv.org/abs/math/#2}{\sf
  arXiv:math/#2
  [math.#1]}}}\providecommand{\MR}[1]{\mbox{\href{http://www.ams.org/mathscinet-getitem?mr=#1}{#1}}}
  \renewcommand{\MR}[1]{\mbox{\href{http://www.ams.org/mathscinet-getitem?mr=#1}{#1}}}
\providecommand{\bysame}{\leavevmode\hbox to3em{\hrulefill}\thinspace}
\providecommand{\MR}{\relax\ifhmode\unskip\space\fi MR }
\providecommand{\MRhref}[2]{%
  \href{http://www.ams.org/mathscinet-getitem?mr=#1}{#2}
}
\providecommand{\href}[2]{#2}

\end{document}